%% file: MV1.tex
\documentclass[11pt]{amsart}
\usepackage{silence}
\ActivateWarningFilters[mybib]
\input{font_language_unicode_pdf.tex}

\input{custom_math_commands.tex}
\usepackage[style=alphabetic,backend=biber]{biblatex}
\input{standard_packages.tex}

\usepackage[marginratio={1:1}]{geometry}
\usepackage{graphicx}
\usepackage{amssymb,amscd}
\usepackage{epstopdf}

\usepackage[justification=centering]{caption}
\usepackage{float}

\definecolor{lightgray}{rgb}{0.9,0.9,0.9}
\usepackage{xcolor}

\usepackage{colortbl}

\usepackage{hyperref}

\usepackage{cleveref}

\crefformat{section}{\S#2#1#3} 
\crefformat{subsection}{\S#2#1#3}
\crefformat{subsubsection}{\S#2#1#3}
\crefformat{theorem}{Theorem \protect{$#2#1#3$}}
\crefformat{cor}{Corollary \protect{\(#2#1#3\)}}

\usepackage{tikz}

\let\tilde\widetilde

\newtheorem{theorem}{Theorem}[section]

\newtheorem{prop}[theorem]{Proposition}
\newtheorem{lemma}[theorem]{Lemma}

\theoremstyle{definition}
\newtheorem{definition}[theorem]{Definition}

\theoremstyle{remark}
\newtheorem{rem}[theorem]{Remark}

\hypersetup{colorlinks=true, linkcolor=blue!70!black, urlcolor=blue!70!black, citecolor=blue!70!black,linktoc=page}

\title{Toward the universal Mumford form on Sato Grassmannians}
\author[K. A. Maxwell]{Katherine A. Maxwell}
\email[K. A. Maxwell]{katherine.maxwell@ipmu.jp}
\address{Kavli IPMU (WPI), UTIAS,
    The University of Tokyo,
    Kashiwa, Chiba 277-8583, Japan
}

\author[A. A. Voronov]{Alexander A. Voronov}
\email[A. A. Voronov]{voronov@umn.edu}
\address{School of Mathematics,
  University of Minnesota,
  Minneapolis, MN 55455,
  and\newline
 Kavli IPMU (WPI), UTIAS,
 The University of Tokyo,
 Kashiwa, Chiba 277-8583, Japan
 }
\thanks{Research is supported in part by World Premier International Research Center Initiative (WPI Initiative), MEXT, Japan. The first author is grateful to Max Planck Institute for Mathematics in Bonn for its hospitality and financial support.}

\dedicatory{Dedicated to the memory of Yuri Ivanovich Manin}

\begin{document}

\begin{abstract}
    We construct a local universal Mumford form on a product of Sato Grassmannians using the flow of the Virasoro algebra. The existence of this universal Mumford form furthers the proposal that the Sato Grassmannian provides a universal moduli space with applications to string theory. Our approach using the Virasoro flow is an alternative to using the KP flow, which in particular allows for a bosonic universal Mumford form to be constructed. Applying the same method, we construct a local universal super Mumford form on a product of super Sato Grassmannians using the flow of the Neveu-Schwarz algebra.
\end{abstract}

\maketitle

\tableofcontents

\section*{Introduction}
The Mumford form is a trivializing section
\begin{equation*}
\mu\co \OO_{\MM_g} \xrightarrow{\sim} \lambda_2 \otimes \lambda_1^{-13},
\end{equation*}
defined up to a constant factor, of the product of determinant line bundles over the moduli space $\MM_g$ of Riemann surfaces of genus $g$. This simple observation of Mumford \cite{Mumford.1977.sopv}, based on the Grothendieck-Riemann-Roch theorem, received considerable publicity after Belavin and Knizhnik \cite{Belavin.Knizhnik.1986.cgattoqs,Beilinson.Manin.1986.tMfatPmist} found out that the Polyakov measure $d π$ in string theory has a simple explicit relation to the Mumford form, which may be compressed to the slogan
\begin{equation*}
d π = \mu \wedge \overline{\mu},
\end{equation*}
see details in \cite{Witten.2019.nosRsatm}. These algebro-geometric or holomorphic methods in string theory opened new possibilities for perturbative computation of the partition function and scattering amplitudes, such as \cite{Beilinson.Manin.1986.tMfatPmist}. Similar results in superstring theory, potentially even more important for high-energy physics, ensued \cite{Voronov.1988.afftMmist,Rosly.Schwarz.Voronov.1989.sgast}.

After a relatively dormant period of some 15 years, D'Hoker and Phong made a breakthrough computation of the amplitudes for the supermoduli space $\MMM_2$ of genus-two super Riemann surfaces \cite{DHoker.Phong.2002.tls,DHoker.Phong.2008.lotls,Witten.2015.nohsastmolg} as well as proposed a computation-friendly expression for the super Mumford form on $\MMM_3$ \cite{DHoker.Phong.2005.amfatsm}, which resulted in partial computation of the amplitudes \cite{Gomez.Mafra.2013.tcs3laaSd}. D'Hoker-Phong's computations were based on ``splitting'' the supermoduli space $\MMM_2$ into the underlying moduli space $\MM_2$ and vector-bundle data on it, and then identifying the moduli space $\MM_2$ of Riemann surfaces with the moduli space $\AAA_2$ of principally polarized abelian varieties of dimension 2. Eyeing possible extension of D'Hoker and Phong's results to higher genera, Donagi and Witten \cite{Donagi.Witten.2015.ssinp} showed that the supermoduli space cannot be split. Moreover, since for higher $g$ the moduli space $\MM_g$ is described as a subspace of $\AAA_g$ via complicated equations (see Shiota's solution \cite{Shiota.1986.coJvitose} of the Schottky problem and also Farkas-Grushevsky-Salvati Manni \cite{Farkas.Grushevsky.SalvatiManni.2021.aesttwSp}), a direct generalization of D'Hoker-Phong's computations to higher genera seems to be out of reach, at least for the time being. There is a modular-form ansatz \cite{Grushevsky.2009.ssaihg} for the odd component of the super Mumford form in arbitrary genus, based on a certain splitting assumption for the super Mumford form and subject to verification of physical constraints, such as the vanishing of the cosmological constant.

The moduli space can also be embedded in the Sato Grassmannian, and one may attempt to use the coordinates on the Grassmannian to describe the Mumford form and perform computations. The Sato Grassmannian, despite being infinite dimensional, may offer a more manageable ambient space for the moduli space, especially given that Shiota's solution of the Schottky problem rested upon Mulase's work \cite{Mulase.1984.csiseaJv}, which characterized the moduli space locus in the Grassmannian via the KP flow. In the super case, a formula for the super Mumford form on a certain subspace of the super Sato Grassmannian was suggested by Schwarz \cite{Schwarz.1987.tfsaaums, Schwarz.1989.fsaums}, who proposed that subspace as the universal (super)moduli space, as it contains the supermoduli space locus. Schwarz's formula used the super KP flow on the super Grassmannian. He also noted that, in the bosonic case, a similar formula would not make sense, as it would be divergent.

The KP flow is not the only way to describe the moduli space locus in the Grassmannian. In the seminal paper \cite{Manin.1986.cdotstatdsotmsosc}, Manin conjectured that the moduli space is an orbit of the Virasoro algebra action on the Grassmannian. This conjecture was proven by Kontsevich \cite{Kontsevich.1987.VaaTs}, Beilinson and Schechtman \cite{Beilinson.Shechtman.1988.dbaVa}, Arbarello, De Concini, Kac, and Procesi \cite{Arbarello.DeConcini.Kac.Procesi.1988.msocart}, and Tsuchiya, Ueno, and Yamada \cite{Tsuchiya.Ueno.Yamada.1989.cftoufoscwgs}. In these papers, a flat holomorphic connection was constructed on the line bundle $\lambda_2 \otimes \lambda_1^{-13}$ over the moduli space $\MM_g$, and the Mumford form was characterized as a horizontal section of that holomorphic connection. Being such, it is defined locally up to a constant factor.

The goal of this paper is to extend the Mumford form, which is originally defined on the moduli space, to other Virasoro orbits on the Grassmannian, and provide a formula for the Mumford form in terms intrinsic to the Grassmannian and independent of the Riemann surface data. In a way, this means that we propose the Grassmannian, or to be more precise, the product $\Gr_2 \times \Gr_1$ of two Grassmannians, as the universal moduli space. This space is larger though simpler than Schwarz's universal moduli space, which is defined by certain nonlinear constraints as a subspace of a single Grassmannian. One may view our work as answering Schwarz's question on constructing a bosonic (non super) analogue of his universal super Mumford form. Our approach uses the Virasoro flow as opposed to the KP flow, and thereby the version of the universal Mumford form which we propose in not only the bosonic, but also fermionic (super) case is quite different from Schwarz's universal super Mumford form. In the super case, the Virasoro flow is actually replaced by the Neveu-Schwarz (NS) flow, based on the first author's work \cite{Maxwell.2022.tsMfaSG}, cf.\ the super KP flow, defined by Manin and Radul \cite{Manin.Radul.1985.aseotKPh} and described as a flow on the super Sato Grassmannian by Mulase \cite{Mulase.1991.ansKsaacotJoaasc} and Rabin \cite{Rabin.1991.tgotsKf}.

\subsection*{Conventions}
We work over the ground field $\CC$ of complex numbers: all super vector spaces, superschemes, etc.\ are assumed to be over $\CC$. By default, we assume our \emph{locally free sheaves} are of constant finite rank and also interchangeably call them \emph{vector bundles}. \emph{Invertible sheaves} or \emph{line bundles} are locally free sheaves of rank $1|0$ or $0|1$. We call them \emph{even} or \emph{odd line bundles}, respectively. We systematically write $=$ for canonical isomorphisms and $\mc{L}^{-1}$ for the dual $\mc{L}^*$ of a line bundle $\mc{L}$.

\section{The Mumford form and the Krichever map}

Understanding the relationship between the operator formulation and the geometrical formulation of (super)string theory was emphasized by Manin in \cite{Manin.1986.cdotstatdsotmsosc}. In this paper, he describes how the constant $6j^2-6j+1$ appears from both the representations of the Virasoro algebra and from the geometry of the moduli space of curves. The constant $6j^2-6j+1$, a multiple of the 2\textsuperscript{nd} Bernoulli polynomial, reduces to the critical dimension of bosonic string theory when $j=2$. Based on the appearance of the value $13$ from the purely algebraic properties of the Virasoro algebra, and the appearance of $13$ in the Mumford isomorphism, Manin suggests that the moduli space of curves plays the role of a homogeneous space with respect to the Virasoro algebra.

Inspired by Manin's paper \cite{Manin.1986.cdotstatdsotmsosc}, several authors simultaneously published papers explaining the story \cite{Arbarello.DeConcini.Kac.Procesi.1988.msocart,Kontsevich.1987.VaaTs,Beilinson.Shechtman.1988.dbaVa,Kawamoto.Namikawa.Tsuchiya.Yamada.1988.grocftoRs}. We briefly review here this story, as explained via the Sato Grassmannian. The analogous story for the case of super Riemann surfaces and the super Mumford form was addressed in \cite{Manin.1988.NSsadefMs, Maxwell.2022.tsMfaSG}.

\subsection{The moduli space of curves and the Witt algebra}

For a relationship to exist between the Virasoro algebra and the moduli space of curves, we hope that that there exists an action by vector fields on the moduli space. This action is the key to understanding the full story with regard to the Mumford isomorphism. Before understanding the action of the Virasoro algebra, which in fact acts on line bundles on the moduli space of curves, we describe the action of the noncentrally extended algebra, the Witt algebra.

\begin{definition}
The Witt algebra $\mathfrak{witt}$ is defined to be the Lie algebra of vector fields on a punctured formal neighborhood of a point in $\mathbb{C}^1$.
\end{definition}
 Written using a local coordinate $z$ at the puncture, $\mathfrak{witt}\cong \mathbb{C}(\!(z)\!)\dd{z}$ and has standard basis $L_n\coeq -z^{n+1}\dd{z}$. By \cite[Proposition 2.1 (2)]{Arbarello.DeConcini.Kac.Procesi.1988.msocart}, it is known that $H^2(\mathfrak{witt})$ is one dimensional. The unique central extension is the Virasoro algebra $\mathfrak{vir}$, with basis $L_n$ for $n\in \mathbb{Z}$ and central element $C$ which satisfy the commutation relations:
\begin{align*}
    [L_m,L_n]=(m-n)L_{m+n}+(m^3-m)\delta_{m,-n}C && [C,L_n]=0.
\end{align*}

The moduli space of Riemann surfaces $\mathcal{M}_g$ does not naturally have an action from the Witt algebra. However, observing that the Witt algebra is naturally identified with a formal punctured neighborhood in $\mathbb{C}^1$ expressed via a coordinate $z$ leads to a natural action on certain decorated moduli spaces of curves. 

Consider triples: a Riemann surface $C$, a marked point $p\in C$, and a formal parameter $z\in \hat{\mathcal{O}}_p$ at $p$.
\begin{definition}
The moduli space of triples $\mathcal{M}_{g,1^k}$ is defined to be the stack which represents the triples $(C,p,z)$ where $C$ is a Riemann surface, a point $p\in C$, and a formal parameter $z\in\hat{\mathcal{O}}_p$ is considered as a $k$-jet equivalence class of a coordinate vanishing at $p$.
\end{definition}
This moduli space $\mathcal{M}_{g,1^k}$ has dimension $3g-3+1+k$, and is representable as a Deligne-Mumford stack.
\begin{definition}
The moduli space of triples $\mathcal{M}_{g,1^\infty}$ is defined to the the pro-Deligne-Mumford stack which is the projective limit over $k$ of the spaces $\mathcal{M}_{g,1^k}$:
\begin{align*}
    \mathcal{M}_{g,1^\infty}\coeq \lim_{\longleftarrow}\mathcal{M}_{g,1^k}
\end{align*}
\end{definition}
We consider the infinite dimensional space $\mathcal{M}_{g,1^\infty}$ to parameterize triples $(C,p,z)$ where now the full infinite information contained in the parameter $z$ is remembered. 

Now, we may describe the action of the Witt algebra on $\mathcal{M}_{g,1^\infty}$ by exploiting the punctured neighborhood on the curves. Consider the universal family $π\co X\to \mathcal{M}_{g,1^k}$, where the marked points of the family are represented by the divisor $P$ of $X$. We have the short exact sequence of sheaves:
\begin{align*}
    0\to \mathcal{T}_{X/\mathcal{M}_{g,1^k}}(-(k+1)P)\to \mathcal{T}_X(-(k+1)P)\to π^*(\mathcal{T}_{\mathcal{M}_{g,1^k}})\to 0.
\end{align*}
In the long exact sequence of higher direct images, the Kodaira-Spencer map is 
\begin{align*}
    \delta\co \mathcal{T}_{\mathcal{M}_{g,1^k}}\overset{\sim}{\to} R^1π_*\mathcal{T}_{X/\mathcal{M}_{g,1^k}}(-(k+1)P)
\end{align*}
Further, relative Čech cohomology based on the set $X\setminus P$ and a formal neighborhood $U$ of $P$ gives the following exact sequence.
\begin{align*}
    π_*\mathcal{T}_{(X\setminus P)/\mathcal{M}_{g,1^k}}\to \frac{π_* \mathcal{T}_{(U\setminus P)/\mathcal{M}_{g,1^k}}}{π_*\mathcal{T}_{U/\mathcal{M}_{g,1^k}}(-(k+1)P)}\to R^1π_*\mathcal{T}_{X/\mathcal{M}_{g,1^k}}(-(k+1)P)
\end{align*}
The projective limit over $k$ then gives
\begin{align*}
    π_*\mathcal{T}_{(X\setminus P)/\mathcal{M}_{g,1^\infty}}\to \mathfrak{witt}\mathrel{\hat{\otimes}}\mathcal{O}_{\mathcal{M}_{g,1^\infty}}\to \mathcal{T}_{\mathcal{M}_{g,1^\infty}}
\end{align*}
where we have identified $π_*\mathcal{T}_{(U\setminus P)/\mathcal{M}_{g,1^\infty}}\cong \mathfrak{witt}\mathrel{\hat{\otimes}}\mathcal{O}_{\mathcal{M}_{g,1^\infty}}$. Taking constant global sections over ${\mathcal{M}_{g,1^\infty}}$, gives the proposition below.

\begin{prop}\label{witt action on moduli space}
The Witt algebra acts on the moduli space $\mathcal{M}_{g,1^\infty}$ by vector fields. Explicitly, there exists a Lie algebra homomorphism
\begin{align*}
    \mathfrak{witt}\to Γ(\mathcal{M}_{g,1^\infty},\mathcal{T}_{\mathcal{M}_{g,1^\infty}}).
\end{align*}
\end{prop}
Intuitively, we see that $L_n$ for $n\leq -2$ change the complex structure of $C$, $L_{-1}$ moves the point $p$ within  $C$, $L_0$ acts by rotating and rescaling the parameter $z$, and $L_n$ for $n\geq 1$ change the parameter $z$ to higher order. On the other hand, it is natural to ask what the stabilizer of this action is. For $(C,p,z)\in \mathcal{M}_{g,1^\infty}$ the Lie subalgebra $Γ(C\setminus p,\mathcal{T}_C)\subset \mathfrak{witt}$ acts by zero at $(C,p,z)$.

For use later on, we need the following result about the stabilizer of the Witt algebra action.
\begin{lemma}[\textcite{Arbarello.DeConcini.Kac.Procesi.1988.msocart}]
\label{perfect lemma}
Let $C$ be a compact Riemann surface and $p\in C$ be a point. Denote by $\mathfrak{k}\coeq Γ(C\setminus p,\mathcal{T}_C)$. Then $\mathfrak{k}$ is perfect, that is $H_1(\mathfrak{k};\mathbb{C})=\mathfrak{k}/[\mathfrak{k},\mathfrak{k}]=0$.
\end{lemma}

\subsection{The Mumford isomorphism}

\begin{definition}[{\textcite{Deligne.1987.lddlc}; \textcite[Definition 5.9]{Mumford.1977.sopv}}]
Let $π\co X \to S$ be a smooth, proper family of complex algebraic curves with $X$ being quasi-projective. Let $\mathcal{F}$ be a locally free sheaf on $X$.
Then the \emph{determinant of cohomology of $\mathcal{F}$} is an invertible sheaf on $S$ given by
\begin{align*}
    D(\mathcal{F}) \coeq \otimes_i\left(\det R^iπ_*\mathcal{F}\right)^{(-1)^i},
\end{align*}
provided the higher direct images $R^iπ_*\mathcal{F}$, $ i \ge 0$, are locally free on $S$.\footnote{If the higher direct images are not locally free, one can still define the determinant of cohomology under these assumptions, see \cite{Deligne.1987.lddlc}.}
We define the \emph{determinant line bundles $\lambda_j$} for the family $π\co X\to S $ as
\begin{align*}
    \lambda_j\coeq D(ω _{X/S})
\end{align*}
where $ω_{X/S }\coeq \Omega_{X/S}^1$.
\end{definition}

If the higher direct images are not locally free, the determinant of cohomology may still be defined: one just needs to use a locally free resolution or twist $\mathcal{F}$ by a relatively sufficiently ample sheaf.

\begin{theorem}[{\textcite[Theorem 5.10]{Mumford.1977.sopv}}]
\label{Mumford-iso}
The Mumford isomorphism is the collection of isomorphisms 
\begin{align*}
    \lambda_j\cong \lambda_1^{\otimes \, 6j^2-6j+1}, \qquad j\in \mathbb{Z},
\end{align*}
of line bundles over $S = \mathcal{M}_g$. Each isomorphism is
defined up to a constant factor.
\end{theorem}
In particular, for $j=2$ the isomorphism $\lambda_2\cong \lambda_1^{\otimes 13}$ is essential to string theory. 
The \emph{Mumford form  $\mu$ on $\MM_g$} is a trivializing section of the isomorphism $\mathcal{O}_{\MM_g} \xrightarrow{\sim} \lambda_2\otimes \lambda_1^{\otimes -13}$, which is defined up to a constant factor. The importance of the Mumford form was made clear in the publications \cite{Belavin.Knizhnik.1986.cgattoqs,Beilinson.Manin.1986.tMfatPmist}, which found that the Polyakov measure $dπ$ of string theory may be expressed using the Mumford form as:
\begin{align*}
    dπ=\mu\wedge\bar{\mu}.
\end{align*}

\subsection{The Sato Grassmannian}

The Sato Grassmannian is a Grassmannian parameterizing semi-infinite subspaces in the precise sense described below of a $\mathbb{Z}$-dimensional vector space. The original definition is due to \textcite{Sato.1981.seadsoaidGm} and originates from the study of soliton equations and the KP equations. Intuitively, the Sato Grassmannian consists of the subspaces of $\mathbb{C}(\!(z)\!)$ close enough to $z^{-1}\mathbb{C}[z^{-1}]$.

We define $H_j\coeq \mathbb{C}(\!(z)\!)\,dz^{\otimes j}$ to be the space of formal Laurent series in $z$, which may be considered a topological vector space under the $z$-adic topology. The distinguished decomposition $H_j\cong H^-_j\oplus H^+_j$ is given by $H^+_j\coeq \mathbb{C}[\![z]\!]\,dz^{\otimes j}$ and $H^-_j\coeq z^{-1}\mathbb{C}[z^{-1}]\,dz^{\otimes j}$.

Since the $H_j$ are isomorphic as topological vector spaces for varying $j$, we may construct the Sato Grassmannian based on any $H_j$ and the geometrical object is the same.

\begin{definition}
\label{commensurable}
Define a subspace $K$ of $H_j$ to be \emph{compact} if it is commensurable with $H^+_j$. Subspaces $K$ and $H^+_j$ are \emph{commensurable} when $(H^+_j + K)/(H^+_j \cap K)$ is finite dimensional.

Define a subspace $D$ of $H_j$ to be \emph{discrete} if there exists a compact subspace $K$ such that the natural map $D\oplus K\to H_j$ is an isomorphism. Equivalently $D$ is discrete iff for each compact subspace $K$, the subspaces $D \cap K$ and $H_j /D + K$ are finite-dimensional.
\end{definition}

\begin{definition}
The \emph{Sato Grassmannian} $\Gr_j$ is the space of all discrete subspaces $D\subset H_j$.
\end{definition}
Here we understand the word ``space'' as the $\CC$-scheme representing a suitable functor of points, see \cite{AlvarezVazquez.MunozPorras.PlazaMartin.1998.tafoseoabf}. This approach is an algebraic formalization of the viewpoint used by Kontsevich \cite{Kontsevich.1987.VaaTs}, who regarded $\Gr_j$ as an ``infinite dimensional manifold'' glued out of copies of infinite dimensional complex affine spaces via concrete rational functions.

\begin{prop}[\textcite{Kontsevich.1987.VaaTs}; \textcite{AlvarezVazquez.MunozPorras.PlazaMartin.1998.tafoseoabf}]
\label{models}
The Sato Grassmannian $\Gr_j$ may be locally modeled on the infinite dimensional affine space $\HHom_\mathbb{C}(H^-_j,H^+_j)$. Namely, $\Gr_j$ admits a covering by affine open \emph{charts} $U_{D,K} \coeq \{ \textup{graphs of $\CC$-linear maps } A\co D \to K \} \cong \HHom (D,K)\cong \HHom (H^-_j,H^+_j)$.
\end{prop}

Note that the local model $\HHom (H^-_j,H^+_j)$ of $\Gr_j$ may be presented as the (inverse) limit
\begin{equation*}
\lim_{K, K' \textup{ compact}} \HHom (H^-_j \cap K , H^+_j/H^+_j \cap K')
\end{equation*}
of finite-dimensional affine spaces.

\begin{definition}
Define the \emph{semi-infinite general linear algebra} $\mathfrak{gl} = \mathfrak{gl}(H_j)$ to be the Lie algebra of continuous linear endomorphisms of $H_j$ with respect to the $z$-adic topology.
\end{definition}

Identification $f(z) dz^{\otimes j} \mapsto f(z)$ defines a canonical continuous linear isomorphism $H_j \xrightarrow{\sim} H_0$, which makes all the Lie algebras $\mathfrak{gl}(H_j)$ canonically isomorphic. This justifies using one notation $\mathfrak{gl}$ for these Lie algebras.

For an endomorphism $F\in \mathfrak{gl}$, we use superscripts to denote the decomposition of the map based on a choice of discrete and compact subspaces. For a decomposition $D\oplus K$ such that the natural map $D\oplus K\to H$ is an isomorphism, we can write the endomorphism $F$ as
\begin{align*}
    F=\begin{pmatrix}
    F^{DD} & F^{DK}\\
    F^{KD} & F^{KK}
    \end{pmatrix}
\end{align*}
where $F^{DK}\co K\to D$ for example. For $D=H^-_j$ and $K=H^+_j$, we use the notation $F^{DK}=F^{-+}$, etc.

\begin{definition} Define the nontrivial $2$-cocycles on $\mathfrak{gl}$
\begin{align*}
    η_{D,K}(F,G)=\tr(F^{DK}G^{KD}-G^{DK}F^{KD})
\end{align*}
where $D\oplus K$ is a decomposition of $H$. In particular, the 2-cocycle
$$η(F,G) = \tr (F^{-+}G^{+-}-G^{-+}F^{+-}),$$
known as the \emph{Japanese cocycle}, generates $H^2(\mathfrak{gl})$. The trace is well-defined, because a continuous linear operator $H^+ \to H^-$ has a finite rank. Define the \emph{central extension} $\tilde{\mathfrak{gl}}$ via the bracket $[F,G]^\sim\coeq [F,G]+η(F,G)C$, where $C$ is a nontrivial central element and $[\;,\;]$ is the bracket on $\mathfrak{gl}$.

\end{definition}

\begin{prop}[{\textcite{Kontsevich.1987.VaaTs}; \textcite[Proposition (1.14)]{Kawamoto.Namikawa.Tsuchiya.Yamada.1988.grocftoRs}}] \label{gl action}
The Lie algebra $\mathfrak{gl}$ acts by vector fields on $\Gr_j$. Specifically, in the chart $U_{D,K}$ the following formula
\begin{align*}
    F\mapsto L_FA\coeq - F^{KD}- F^{KK}A+AF^{DD}+AF^{DK}A.
\end{align*}
defines a Lie algebra homomorphism
\begin{align*}
    \mathfrak{gl}\to Γ({\Gr_j},\mathcal{T}_{\Gr_j}).
\end{align*}

\end{prop}

\begin{definition}
The \emph{determinant line bundle $\det_j$ over the Sato Grassmannian} $\Gr_j$  is defined as the line bundle with fiber over $D\in \Gr_j$ given by
\begin{align*}
    \det_j|^{}_D\coeq \frac{\det(D\cap H^+_j)}{\det(H_j/(D+H^+_j))}.
\end{align*}
\end{definition}

One may expect that the action of $\mathfrak{gl}$ on the Grassmannian extends to an action of the determinant line bundle $\det_j$. Indeed, the action extends but up to a scalar factor, that is to say, to an action of the central extension $\tilde{\mathfrak{gl}}$.

A Lie algebra action on a line bundle is explicitly a Lie algebra homomorphism to the Lie algebra of (global sections of) first-order differential operators acting on the line bundle. We denote the sheaf of first-order differential operators on a line bundle $L$ by $\mathcal{A}_L$, since this sheaf is equivalently known as the Aityah Lie algebroid of $L$. Recall that the Atiyah Lie algebroid fits into the short exact sequence (of holomorphic sheaves and of Lie algebras)
\begin{equation}
\label{symbol}
    0\to \mathcal{O}_X\to \mathcal{A}_L\xrightarrow{\sym} \mathcal{T}_X\to 0,
\end{equation}
where the anchor map is the \emph{symbol map}: $\sym(D)(df)=[D,f]$. The Atiyah Lie algebroid is trivial iff there exists a holomorphic flat connection $\mathcal{T}_X\to \mathcal{A}_L$ which splits the above sequence. 

Locally we always have, $\mathcal{A}_L\cong \mathcal{O}_X\oplus \mathcal{T}_X$, and so locally we denote a differential operator as $(D,d)\in \mathcal{A}_L(U)$ which acts as $(D(x),d(x))s(x)=\lim\limits_{\epsilon\to 0}\frac{s(x+\epsilon D(x))-s(x)}{\epsilon}+d(x)s(x)$.

\begin{prop}[{\textcite{Kontsevich.1987.VaaTs}}]\label{gl central extension action} \label{gl tilde action}
The Lie algebra $\tilde{\mathfrak{gl}}$ acts by first-order differential operators on $\det_j$. Specifically, there exists a Lie algebra homomorphism
\begin{align*}
    \tilde{\mathfrak{gl}}\to Γ(\Gr_j,\mathcal{A}_{\det_j})
\end{align*}
with the central generator $C$ mapping to $1\inΓ(\Gr_j,\mathcal{O}_{\Gr_j})$.
This action is given in the chart $U_{D,K}$ by the formula
\begin{align*}
    F\mapsto \tilde{L}_FA\coeq \Big(-F^{KD}-F^{KK}A+AF^{DD}+AF^{DK}A,\; -\tr(F^{DK}A)-\alpha(F)\Big),
\end{align*}
where $\alpha\in C^1(\mathfrak{gl})$ is the unique $1$-cochain such that
\begin{align*}
    d\alpha(F,G)=\alpha([F,G])=η_{D,K}(F,G)-η(F,G).
\end{align*}
\end{prop}

\subsection{The Krichever map and the Virasoro algebra} \label{krichever subsection}

The Krichever map is the geometrical map between the moduli space of triples $\mathcal{M}_{g,1^\infty}$ and the Sato Grassmannian $\Gr_j$. The compatibility of this geometrical map with the Lie algebras $\mathfrak{witt}$ and $\mathfrak{gl}$ and their respective actions allows us to properties of study line bundles on $\mathcal{M}_{g,1^\infty}$. In particular, we explain in this section that via the $\mathfrak{vir}$ action there exists a flat holomorphic connection on $\lambda_2\otimes \lambda_1^{-13}$.

The extra information contained in the puncture $p$ and parameter $z$, allows for a natural map from a Riemann surface into the Sato Grassmannian. The key observation is that the sections of any line bundle on $C$ taken over the curve minus the point $p$ are a discrete subspace expressed in the formal coordinate $z$.
\begin{definition}[{\textcite[Proposition 6.2]{Segal.Wilson.1985.lgaeoKt}}]
The Krichever map
\begin{align*}
	\kappa_j\co \mathcal{M}_{g,1^\infty}\to \Gr_j
\end{align*}
is defined by
\begin{align*}
    \kappa_j(C,p,z)\coeq Γ(C\setminus p,ω_C^{\otimes j}) \subset \mathbb{C}(\!(z)\!)\,dz^{\otimes j}
\end{align*}
where $ω_C\coeq \Omega_C^1$.
\end{definition}

On the algebraic side of the story, notice that the Witt algebra has natural interpretation as endomorphisms of $H_j\coeq \mathbb{C}(\!(z)\!)\,dz^{\otimes j}$ based on the Lie derivative. Since the Lie derivative action on $j$-differentials is different depending on the value of $j$, different representations of the Witt algebra are produced despite the fact that the $H_j$'s are isomorphic as topological vector spaces.
\begin{definition}\label{representations definition}
The formula
\begin{align*}
    \rho_j\co\mathfrak{witt}&\to \mathfrak{gl} = \mathfrak{gl}(H_j),\\
    f(z)\dd{z}&\mapsto \bigg(g(z)\,dz^{\otimes j}\mapsto \Big(f(z)g'(z)+jf'(z)g(z)\Big)\,dz^{\otimes j}\bigg)
\end{align*}
defines a natural representation\footnote{These representations form what is called the \emph{intermediate series}.} of the Witt algebra on $H_j$ for each $j \in \ZZ$.
\end{definition}

The appearance of the critical dimension $26 = 13 \cdot 2$ from the representation theory of the Virasoro algebra is essentially the calculation below for $j=2$.
\begin{prop}[{\textcite[(2.24)]{Arbarello.DeConcini.Kac.Procesi.1988.msocart}}]\label{pullback critical number}
The pullbacks along the representations $\rho_j$ satisfy
\begin{align*}
    \rho_j^*(η)=(6j^2-6j+1)\rho_1^*(η)
\end{align*}
\end{prop}
Using this relationship we can establish representations of $\mathfrak{vir}$ on $\tilde{\mathfrak{gl}}$ as given by the central arrow in the commutative diagram below. The key observation is that $\rho_1^*(η)$ is the standard nontrivial cocycle which defines $\mathfrak{vir}$. Then by \cref{pullback critical number} the cocycle $\rho_j^*(η)$ is the $c_j = 6j^2-6j+1$ multiple of the standard cocycle.
\begin{equation}\label{vir representations}
    \begin{tikzcd}
    0\arrow{r}& \mathbb{C} \arrow{d}{\cdot c_j}\arrow{r}&\mathfrak{vir}\arrow{r}\arrow{d}{\tilde{\rho}_j}&\mathfrak{witt}\arrow{r}\arrow{d}{\rho_j}&0\\
    0\arrow{r}& \mathbb{C} \arrow{r}&\tilde{\mathfrak{gl}}\arrow{r}&\mathfrak{gl}\arrow{r}&0
    \end{tikzcd}
\end{equation}

We now return to the geometrical properties of the Krichever maps.

By construction based on the $j$-differentials in the Krichever maps $\kappa_j$ and the representations $\rho_j$, the differentials of the Krichever maps are compatible with the representations of $\mathfrak{witt}$.
\begin{equation*}
    \begin{tikzcd}
    \mathfrak{witt}\arrow{r}\arrow{d}{\rho_j}&Γ(\mathcal{M}_{g,1^\infty},\mathcal{T}_{\mathcal{M}_{g,1^\infty}})\arrow{d}{d\kappa_j}\\
    \mathfrak{gl}\arrow{r}&Γ(\mathcal{M}_{g,1^\infty},\kappa_j^*\mathcal{T}_{\Gr_j})
    \end{tikzcd}
\end{equation*}
Therefore, we can identify the moduli space of triples for a fixed genus as an orbit on $\Gr_j$ of the Witt algebra under the $\rho_j$ representation.

The definition of the Krichever map also naturally identifies the determinant line bundles $\lambda_j$ as the pullbacks of the determinant line bundle $\det_j$ on the Grassmannian.
\begin{prop}[\textcite{Kontsevich.1987.VaaTs}]
\label{kappa}
The pullback of the determinant line bundle from the Sato Grassmannian along $j$\textsuperscript{th} Krichever map is canonically isomorphic to $\lambda_j$.
\begin{align*}
    \kappa_j^*\det_j \cong \lambda_j
\end{align*}
\end{prop}

Combining the representations $\tilde{\rho}_j$ defined by \eqref{vir representations}, the geometrical compatibility of the Krichever maps above, and the action of $\tilde{\mathfrak{gl}}$ as in \cref{gl tilde action}, we arrive at the action on $\lambda_j$ as below.
\begin{prop}
The Virasoro algebra acts by first-order differential operators on $\lambda_j$
\begin{align*}
    \mathfrak{vir}\to Γ(\mathcal{M}_{g,1^\infty},\mathcal{A}_{\lambda_j})
\end{align*}
where the central generator $C$ maps to $c_j\in Γ(\mathcal{M}_{g,1^\infty},\mathcal{O}_{\mathcal{M}_{g,1^\infty}})$.
\end{prop}

Using the value of the central charge action from \cref{pullback critical number}, we can see that on certain tensor products of line bundles the action of the central charge is by zero. For such a line bundle, the Virasoro action descends to a Witt algebra action.
\begin{lemma}
The Witt algebra acts by first-order differential operators on $\lambda_j\otimes\lambda_1^{-c_j}$.
\begin{align*}
    \mathfrak{witt}\to Γ(\mathcal{M}_{g,1^\infty},\mathcal{A}_{\lambda_j\otimes\lambda_1^{-c_j}})
\end{align*}
\end{lemma}

This Witt algebra action on the moduli space of triples satisfies: the action is surjective on the tangent space to the moduli space of triples, and the kernel of the action is a perfect Lie algebra as stated in \cref{perfect lemma}. From these two properties, the line bundles $\lambda_j\otimes\lambda_1^{-c_j}$ are seen to be holomorphically flat.

\begin{theorem}[{\textcite{Kontsevich.1987.VaaTs}; \textcite{Arbarello.DeConcini.Kac.Procesi.1988.msocart}}]
There exists a flat holomorphic connection on the line bundle $\lambda_j\otimes\lambda_1^{-c_j}$.
\end{theorem}

\section{The universal Mumford form on a Witt-algebra orbit} 
\label{mumford on witt orbit section}

\subsection{The Witt and Virasoro action Lie algebroids}
\label{algebroids}

In this subsection, we define an action Lie algebroid. Lie algebroids are the natural generalization of Lie algebras to families and therefore carry a sheaf structure. 

\begin{definition}[Action Lie algebroid] \label{action lie algebroid}
Let $\mathfrak{g}$ be a Lie algebra acting on a variety $M$ via an
 infinitesimal action map $\mathfrak{g}\to Γ(TM)$, $X \mapsto
 \xi_X$. Then $\mathfrak{g}\times M\to M$ is a Lie algebroid with
 $(X,m)\mapsto \xi_X (m)$ as the anchor and the bracket defined
 pointwise by
\begin{equation*}
 [X,Y](m) \coeq [X(m),Y(m)]_{\mathfrak{g}} + \xi_X Y (m)-\xi_YX (m),
\end{equation*}
where a section $X$ of $\mathfrak{g} \times M \to M$ is identified
with a function $X\co M \to \mathfrak{g}$.
\end{definition}

\begin{prop}[{\textcite[Theorem 2.4]{KosmannSchwarzbach.Mackenzie.2002.doaaoLa}}]
    Consider the action Lie algebroid $\mathcal{G}$ associated to a Lie algebra morphism $ϕ \co \mathfrak{g} \to
Γ( S , T S )$. Separately consider a line bundle $L$ over $S$. There is a bijection between:
\begin{itemize}
    \item Lie algebra morphisms $\mathfrak{g} \to Γ( S , \mathcal{A}_L )$,
    \item Lie algebroid morphisms $\mathcal{G} \to \mathcal{A}_L$,
\end{itemize}
where $\mathcal{A}_L$ is the Atiyah Lie algebroid of $L$.
\end{prop}

We construct two action Lie algebroids over the Sato Grassmannian $\Gr_j$.

\begin{definition}
       According to the construction described in \cref{action lie algebroid}, define the \emph{Witt action Lie algebroid} $\mathcal{W}_{\Gr_j}$  based on the Lie algebra morphism 
       \begin{equation*}
           \begin{tikzcd}
               \mathfrak{witt} \arrow{r}{ρ_j} & \mathfrak{gl} \arrow{r}{L} &  Γ(\Gr_j,\mathcal{T}_{\Gr_j}),
           \end{tikzcd}
       \end{equation*}
       which is the composition of the representation of $\witt$ in \cref{representations definition} with the action of $\mathfrak{gl}$ by vector fields in \cref{gl action}.

        Similarly, define the \emph{Virasoro action Lie algebroid} $\mathcal{V}_{\Gr_j}$  based on the Lie algebra morphism 
        \begin{equation*}
           \begin{tikzcd}
               \mathfrak{vir} \arrow{r}{\tilde{ρ}_j} & \tilde{\mathfrak{gl}} \arrow{r}{\tilde{L}} &  Γ(\Gr_j,\mathcal{A}_{\det_j}) \arrow{r}{\sym} & Γ(\Gr_j,\mathcal{T}_{\Gr_j}),
           \end{tikzcd}
       \end{equation*}
        which is the composition of the representation defined by Diagram \cref{vir representations} with the action of $\tilde{\mathfrak{gl}}$ by first-order differential operators in \cref{gl tilde action} and the symbol map \eqref{symbol} of the Atiyah Lie algebroid.
\end{definition}

We may summarize the action of the Wit and Virasoro action Lie algebroids on the Sato Grassmannian into one diagram:
\begin{equation*}
    \begin{tikzcd}
    0\arrow{r}&\mathcal{O}_{\Gr_j}\arrow{r}\arrow{d}{\cdot c_j} &\mathcal{V}_{\Gr_j}\arrow{r}\arrow{d}&\mathcal{W}_{\Gr_j}\arrow{r}\arrow{d}&0\\
    0\arrow{r}&\mathcal{O}_{\Gr_j}\arrow{r}&\mathcal{A}_{\det_j}\arrow{r}&\mathcal{T}_{\Gr_j}\arrow{r}&0.
    \end{tikzcd}
\end{equation*}
The action of the center as multiplication by $c_j=6j^2-6j+1$ follows from the result in \cref{pullback critical number} about the representations $ρ_j$.

\subsection{A flat holomorphic connection}

Consider the product $\dGr$ of two Sato Grassmannians, with natural projections $p_1\co \dGr \to \Gr_1$ and $p_2\co \dGr \to \Gr_2$. Recall the \emph{external tensor product of vector bundles} $E_1$ over $\Gr_1$ and $E_2$ over $\Gr_2$:
\begin{equation*}
E_2 \boxtimes E_1 \coeq p_2^* (E_2) \otimes p_1^* (E_1) .
\end{equation*}
\begin{prop}
\label{Witt-equiv}
The line bundle $\det_2 \boxtimes \det_1^{-13}$ on $\dGr$ is $\mathfrak{witt}$-equivariant.
\end{prop}
\begin{proof}
Using the representations $\tilde{\rho}_j\co \mathfrak{vir}\to \tilde{\mathfrak{gl}}$ as in \cref{vir representations} and the action of $\tilde{\mathfrak{gl}}$ as in \cref{gl central extension action}, there is a natural action of $\mathfrak{vir}$ by first-order differential operators on $\det_j$.

Under the representation $\tilde{\rho}_j$ of $\mathfrak{vir}$, the central charge acts by $c_j\inΓ(\Gr_j,\mathcal{O}_{\Gr_j})$ on the line bundle $\det_j$. In particular, for $j=1$, the action of $\mathfrak{vir}$ on $\det_1$ has central charge $1\inΓ(\Gr_1,\mathcal{O}_{\Gr_1})$, and therefore the central charge acts by $-c_j\inΓ(\Gr_1,\mathcal{O}_{\Gr_1})$ on the line bundle $\det_1^{-c_j}$. In particular, the central charges for $\det_j$ and $\det_1^{-c_j}$ cancel on $\det_j  \boxtimes \det_1^{-c_j}$. Thus, for $j=2$, we have a commutative diagram of Lie algebroids, with the dashed arrow lifting an action of the Lie algebra $\mathfrak{witt}$ from $\dGr$ to the line bundle $\det_2 \boxtimes \det_1^{-13}$:

\begin{equation}
\label{doubleAtiyah}
    \begin{tikzcd}
    0\arrow{r}&\mathcal{O}_{\dGr}\arrow{r}\arrow{d}{\cdot 0} &\mathcal{V}_{\dGr}\arrow{r}\arrow{d}&\mathcal{W}_{\dGr}\arrow{r}\arrow{d}{\alpha} \arrow[dashed]{dl}&0 \\ 0\arrow{r}&\mathcal{O}_{\dGr}\arrow{r}&\mathcal{A}_{\det_2 \boxtimes \det_1^{-13}}\arrow{r}&\mathcal{T}_{\dGr}\arrow{r}&0. 
    \end{tikzcd}
\end{equation}

\vspace*{-2em} \qedhere

\end{proof}

The Krichever maps $\kappa_2\co \mathcal{M}_{g,1^\infty}\to \Gr_2$ and $\kappa_1\co \mathcal{M}_{g,1^\infty}\to \Gr_1$ define a ``diagonal'' Krichever map
\begin{equation*}
(\kappa_2, \kappa_1)\co \mathcal{M}_{g,1^\infty}\to \Gr_2 \times \Gr_1,
\end{equation*}
which is $\mathfrak{witt}$-equivariant, and thereby, its image is a $\witt$-orbit, which we call the \emph{moduli space orbit} or \emph{locus}.

\begin{theorem}
\label{main}
The line bundle $\det_2 \boxtimes \det_1^{-13}$ over a $\witt$-orbit near the moduli space locus on $\dGr$ has a flat holomorphic connection.
\end{theorem}

\begin{proof}
Let $M$ be a $\witt$-orbit in $\dGr$ and $\mathcal{L}$ be the restriction of the line bundle $\det_2 \boxtimes \det_1^{-13}$ to $M$.
By \cref{algebroids}, we need to find a Lie-algebroid splitting of the short exact sequence
\begin{equation*}
0 \to \mathcal{O}_{M}\to \mathcal{A}_{\mathcal{L}}\to \mathcal{T}_{M} \to 0,
\end{equation*}
that is to say, a Lie-algebroid morphism $\mathcal{T}_{M} \to \mathcal{A}_{\mathcal{L}}$.

On the orbit $M$, we have the following diagram, built on diagram \Cref{doubleAtiyah},
\begin{equation}
\label{orbitAtiyah}
    \begin{tikzcd}
        &&&0\arrow{d}&\\
    &&&\mathcal{K}_M\arrow{d}&\\
    &&&\mathcal{W}_{M}\arrow{r}\arrow{d}{\alpha_M} \arrow[dashed]{dl}&0\\
0\arrow{r}&\mathcal{O}_{M}\arrow{r}&\mathcal{A}_{\mathcal{L}}\arrow{r}[near end]{\beta} &\mathcal{T}_{M}\arrow{r}\arrow{d}&0\\
        &&&0
    \end{tikzcd}
\end{equation}
with exact row and column, where $\alpha_M$ is surjective by definition and $\mathcal{K}_M \coeq \ker \alpha_M$. All arrows are morphisms of Lie algebroids. To split $\beta$, it is enough to show that $\mathcal{K}_M$ maps to 0 in $\mathcal{A}_{\mathcal{L}}$. Since $\mathcal{K} = \ker \alpha_M$, the image of $\mathcal{K}_M$ in $\mathcal{A}_{\mathcal{L}}$ is contained in $\mathcal{O}_{M}$. The Lie algebroid $\mathcal{O}_{M}$ is abelian, therefore, it is enough to show $\mathcal{K}_M$ is perfect, \emph{i.e}., $\mathcal{K}_M = [\mathcal{K}_M, \mathcal{K}_M]$. For a $\witt$-orbit sufficiently close to the moduli space orbit, this fact is a consequence of the following lemma.

\begin{lemma}
\label{locally-perfect}
The stabilizer $\mathcal{K}_M$ of a $\witt$-orbit $M$, in a neighborhood sufficiently close to a point of the moduli space locus in $\Gr_2 \times \Gr_1$,  is perfect.
\end{lemma}

\begin{proof}[Proof of Lemma]
 The question is local on $\Gr_2 \times \Gr_1$, so we will need to take diagram \cref{orbitAtiyah} back there:
 \begin{equation*}
    \begin{tikzcd}
        &&&0\arrow{d}&\\
    &&&\mathcal{K}\arrow{d}&\\
    &&&\mathcal{W}_{\Gr_2 \times \Gr_1}\arrow{r}\arrow{d}{\alpha} \arrow[dashed]{dl}&0\\
0\arrow{r}&\mathcal{O}_{\Gr_2 \times \Gr_1}\arrow{r}&\mathcal{A}_{\det_2 \boxtimes \det_1^{-13}}\arrow{r}
{\beta} &\mathcal{T}_{\Gr_2 \times \Gr_1}\arrow{r}&0.
    \end{tikzcd}
\end{equation*}
Here, as before, $\mathcal{K} \coeq \ker \alpha$ is the sheaf of stabilizers of the Lie-algebra action of $\mathfrak{witt}$ on $\Gr_2 \times \Gr_1$. Since $\alpha$, being an anchor, is a morphism of vector bundles, $\mathcal{K}$ is a sheaf of $\mathcal{O}_{\Gr_2 \times \Gr_1}$-modules. The vanishing of the anchor map on $\mathcal{K}$ implies that the bracket on it is $\mathcal{O}_{\Gr_2 \times \Gr_1}$-linear. Therefore, the derived Lie algebroid $[\mathcal{K},\mathcal{K}]$ and the quotient $\mathcal{K}/[\mathcal{K},\mathcal{K}]$ are $\mathcal{O}_{\Gr_2 \times \Gr_1}$-modules.

On the other hand, \Cref{perfect lemma} states that the fiber $\mathcal{K}(m)$ of $\mathcal{K}$ over a point $m$ of the moduli space locus in $\Gr_2 \times \Gr_1$ is perfect, \emph{i.e}., $\mathcal{K}(m)/[\mathcal{K}(m),\mathcal{K}(m)] = 0$. This yields $\mathcal{K}/[\mathcal{K},\mathcal{K}] = 0$ locally near $m$ by Nakayama's lemma. Nakayama's lemma is applicable without the standard finiteness condition,
because the Sato Grassmannian is locally a limit of finite-dimensional affine spaces via affine morphisms, see the remark after \cref{models}.
Thus, on a $\witt$-orbit $M$, in an open neighborhood sufficiently close to $m$, we have $\mathcal{K}_M / [\mathcal{K}_M, \mathcal{K}_M] = 0$.
Lemma is proven.
\end{proof} 
\let\qed\relax
\end{proof}

\subsection{The local universal Mumford form}

\Cref{main} justifies the following definition, which defines a universal Mumford form on a $\witt$-orbit locally. Note that the original Mumford form, which is a section establishing the Mumford isomorphism of \Cref{Mumford-iso}, is defined up to a constant factor only globally over the moduli space $\mathcal{M}_g$. This is because of the computation $Γ(\mathcal{M}_g, \mathcal{O}^*_{\mathcal{M}_g}) = \CC^*$, see \cite{Mumford.1977.sopv}. Locally on $\mathcal{M}_g$, the Mumford isomorphism and form are defined only up to an invertible function. The universal Mumford form defined below on a $\witt$-orbit is defined locally up to a constant factor. In particular, the universal Mumford form is defined on the moduli space orbit $M = \mathcal{M}_{g,1^\infty}$ locally up to a constant factor.

\begin{definition}
A \emph{universal Mumford form} $\mu_M$  on a $\witt$-orbit $M$ sufficiently close to the moduli space locus in $\dGr$ is a local horizontal section of $\mathcal{L} = (\det_2 \boxtimes \det_1^{-13})|_M$.
\end{definition}

The construction of this section justifies the product $\Gr_2 \times \Gr_1$ of Sato Grassmannians as a universal moduli space. The idea of including all the moduli spaces $\MM_g$ within a ``universal moduli space'' originated in the work \cite{Friedan.Shenker.1986.tiagoqs}. \textcite{Manin.1986.cdotstatdsotmsosc} opened the door to using the Sato Grassmannian as the universal moduli space by conjecturing its relevance to the Virasoro algebra and KP hierarchy. The Sato Grassmannian was explicitly suggested as a universal moduli space in \cite{AlvarezGaume.Gomez.Reina.1987.lgGast, Kontsevich.1987.VaaTs, Morozov.1987.statsoums, Vafa.1987.ofoRs, Arbarello.DeConcini.Kac.Procesi.1988.msocart}. \textcite{Schwarz.1987.tfsaaums, Schwarz.1989.fsaums} proposed a certain subspace of the Grassmannian as the universal moduli space and constructed a super Mumford form on it in the fermionic case.

\begin{rem}[Descent to $\MM_g$] The moduli space orbit $\MM_{g,1^\infty}$ is infinite dimensional but admits a natural forgetful projection to the finite-dimensional moduli space $\MM_g$:
\begin{equation*}
p\co \MM_{g,1^\infty} \to \MM_g,
\end{equation*}
compatible with the tensor products of the determinant line bundles:
\begin{equation*}
(\det_2 \boxtimes \det_1^{-13})|_{\MM_{g,1^\infty}}  \cong p^* (\lambda_2 \otimes \lambda_1^{-13}),
\end{equation*}
just because $\kappa_j^* \det_j \cong p^* \lambda_j$. The flat holomorphic connection on $p^* (\lambda_2 \otimes \lambda_1^{-13})$ induces a flat holomorphic connection on $\lambda_2 \otimes \lambda_1^{-13}$. The local universal Mumford form $\mu_M$ on the $\witt$-orbit $M = \MM_{g,1^\infty}$, being horizontal with respect to the flat holomorphic connection, is a pullback of a horizontal local section $\mu$ of a flat holomorphic connection on $\lambda_2 \otimes \lambda_1^{-13}$ over $\MM_g$:
\begin{equation*}
\mu_M = p^* \mu.
\end{equation*}
The original Mumford, which is itself defined up to a constant factor, is also horizontal with respect to this flat holomorphic connection and thereby coincides with $\mu$ everywhere where $\mu$ is defined, up to a constant factor.
\end{rem}

\section{A formula for the Mumford form}
\label{expformula}

Let $M$ be a Witt-algebra orbit on $\dGr$, $m \in M$ be a point on the orbit sufficiently close to a point on the moduli space locus, $\K_m = \ker \alpha_m \subset \witt$ be its stabilizer, and $\{ \DD_1, \DD_2, \dots \}$ a basis of the quotient $\witt/\K_m$. Locally, in the formal neighborhood of $m$ on $M$, one can use formal logarithmic coordinates $t_1, t_2, \dots$:
\begin{equation*}
\exp \left(\sum_{n=1}^\infty t_n \DD_n\right) m.
\end{equation*}
The Mumford form $\mu_M$ may be expressed as
\begin{equation*}
\mu_M (t_1, t_2, \dots)  = \exp \left(\sum_{n=1}^\infty t_n \DD_n\right) \mu_m,
\end{equation*}
where $\mu_m$ is a nonzero vector in the fiber of the line bundle $\det_2 \boxtimes \det_1^{-13}$ over $m$. This expression is well-defined, because the elements of $\K_m$ act trivially on the fiber of this line bundle over $m$. This follows from \Cref{locally-perfect}. Thus, the Mumford form $\mu_M$ is defined locally uniquely up to a constant factor. Note that the classical Mumford form $\mu$ on the moduli space $\MM_g$ is defined globally uniquely up to a constant factor. This follows from the computation $H^0(\MM_g, \OO_{\MM_g}) = \CC$, see \cite{Mumford.1977.sopv}.

By construction, in the formal neighborhood of the point $m$, the formal variables $t_1, t_2, \ldots$ form a system of formal coordinates. However, it might be useful to employ simpler ``coordinates,'' changing along the vector fields $L_n = -z^{n+1}\dd{z}$, quotation marks emphasizing the variables $s_n$ are not independent: $\dots, s_{-1}, s_0, s_1, \dots$ A point in the Witt-orbit $M$ is given by
\begin{equation*}
\exp \left(\sum_{n=-\infty}^\infty s_n L_n\right) m,
\end{equation*}
whereas the Mumford form may be expressed as
\begin{equation*}
\mu_M (\dots, s_{-1}, s_0, s_1, \dots)  = \exp \left(\sum_{n=-\infty}^\infty s_n L_n\right) \mu_m.
\end{equation*}

\section{The super Mumford form}

 In the fermionic case, when we deal with the supermoduli space $\MMM_g$ on the one hand and the super Grassmannian on the other, the constructions and results are quite parallel to the bosonic case, even though they are not straightforward generalizations of their bosonic counterparts. Here is a dictionary of the basic notions, known since \cite{Manin.1986.cdotstatdsotmsosc, Manin.1988.NSsadefMs}. One may find details of the super column in \cite{Maxwell.2022.tsMfaSG}. The super Heisenberg group $Γ$ may be defined by its functor of points, similarly to \cite{AlvarezVazquez.MunozPorras.PlazaMartin.1998.tafoseoabf, MunozPorras.PlazaMartin.2008.eoHsitiG} in the bosonic case.

\begin{table}[H]
\begin{tabular}{llll}
\hline
\multicolumn{2}{c}{\textbf{Classical case}} &   \multicolumn{2}{c}{\textbf{Super case}}\\
\hline
 \hline
\rowcolor{lightgray} algebraic curve & $C$ &  super Riemann surface & $\Sigma$ \\
moduli space &$\MM_g$ & supermoduli space &$\MMM_g$ \\
\rowcolor{lightgray} Mumford isomorphism &$\lambda_2 \cong \lambda_1^{13}$ & super Mumford isomorphism & $\lambda_{3/2} = \lambda_{1/2}^{5}$ \\
Witt algebra &$\witt$  & super Witt algebra &$\switt$\\
\rowcolor{lightgray} Virasoro algebra &$\vir$ & Neveu-Schwarz algebra &$\ns$ \\
formal Laurent series & $H_j$ & formal super Laurent series & $H_{j/2}$\\
\rowcolor{lightgray} Heisenberg group &$Γ$ &  super Heisenberg group &$Γ$  \\
Grassmannian & $\Gr_j$ & super Grassmannian &  $\Gr_{j/2}$\\
determinant line bundle & $\det_j$ & Berezinian line bundle & $\Ber_{j/2}$\\
\rowcolor{lightgray}universal moduli space &$\Gr_2 \times \Gr_1$ & universal supermoduli space &$\Gr_{3/2} \times \Gr_{1/2}$   \\
Mumford form bundle&$\det_2\boxtimes\det^{-13}_1$ & super Mumford form bundle&$\Ber_{3/2}\boxtimes\Ber^{-5}_{1/2}$\\
\hline
\end{tabular}
\caption{The classical/super dictionary.}
\end{table}

Here we outline the super versions of the main results of \cref{mumford on witt orbit section} and \cref{expformula}. The proofs of the statements below are direct generalizations of the nonsuper results.

\begin{prop}
The line bundle $\Ber_{3/2}\boxtimes\Ber^{-5}_{1/2}$ on $\sdGr$ is $\switt$-equivariant.
\end{prop}

\begin{theorem}
\label{fmain}
The line bundle $\Ber_{3/2}\boxtimes\Ber^{-5}_{1/2}$ over an $\switt$-orbit near the supermoduli space locus on $\sdGr$ has a flat holomorphic connection.
\end{theorem}

\begin{lemma}
The stabilizer $\mathcal{K}_M$ of an $\switt$-orbit $M$, in a neighborhood sufficiently close to a point of the supermoduli space locus in $\sdGr$,  is perfect.
\end{lemma}

Our semicontinuity/Nakayama's lemma approach to the proof of the bosonic version of this lemma, \cref{locally-perfect}, which goes through verbatim in the fermionic case, gives a new, shorter proof of the particular case of this lemma which deals with the supermoduli locus and which was proven originally in \cite[Proposition 3.6]{Maxwell.2022.tsMfaSG}. The original proof used a superconformal generalization of the Noether normalization lemma to a family of affine super Riemann surfaces \cite[Lemma 3.5]{Maxwell.2022.tsMfaSG}.

\cref{fmain} allows us to define a universal super Mumford form on an $\switt$-orbit in a similar way as in the bosonic case.

\begin{definition}
A \emph{universal super Mumford form} $\mu_M$  on an $\switt$-orbit $M$ sufficiently close to the supermoduli space locus in $\sdGr$ is a local horizontal section of $\mathcal{L} = (\Ber_{3/2}\boxtimes\Ber^{-5}_{1/2})|_M$.
\end{definition}

We also have an exponential formula for the universal super Mumford form, analogous to one in \cref{expformula}:
\begin{equation*}
\mu_M (\dots, s_{-1}, s_0, s_1, \dots | \dots, \sigma_{-1/2}, \sigma_{1/2}, \sigma_{3/2}, \dots )  = \exp \left(\sum_{n \in \ZZ} s_n L_n + \sum_{r \in \tfrac{1}{2} + \ZZ} \sigma_{r} G_r \right) \mu_m,
\end{equation*}
where $m$ is a point in $\sdGr$ sufficiently close to the supermoduli space locus, $s_n$'s are even variables and $\sigma_r$'s are odd variables along the $\switt$ orbit of $m$, and
\begin{align*}
L_n & = -z^{n+1} \frac{\partial}{\partial z} -\frac{n+1}{2} z^{n} ζ \frac{\partial}{\partial ζ}, & n \in \ZZ,\\
G_r & = -\frac{1}{2}z^{r+\tfrac{1}{2}} \left( ζ \frac{\partial}{\partial z} - \frac{\partial}{\partial ζ} \right), & r \in \frac{1}{2} + \ZZ,
\end{align*}
is the standard basis of the super Witt algebra, see \cite{Kac.Leur.1989.ocosa} and, for example, \cite[Definition 3.8]{Maxwell.2022.tsMfaSG}.

\printbibliography

\end{document}

%% file: font_language_unicode_pdf.tex
\usepackage{amsmath}

\DeclareMathAlphabet{\mathup}{OT1}{\familydefault}{m}{n}

\usepackage{amssymb}

\usepackage[utf8]{inputenc}
\usepackage[OT1,T2A]{fontenc}
\usepackage[whole]{bxcjkjatype}
\usepackage{alphabeta}

\usepackage{fixmath}

\usepackage{newunicodechar}
\newunicodechar{“}{``}
\newunicodechar{”}{''}
\newunicodechar{⸨}{\mathopen{(\!(}}
\newunicodechar{⸩}{\mathclose{)\!)}}

\usepackage[main=english,russian,provide*=*]{babel}
\usepackage{csquotes}
\usepackage{CJKutf8}

%% file: custom_math_commands.tex
\usepackage{amsmath,color}
\let\epsilon\varepsilon
\let\phi\varphi

\let\hat\widehat
\let\tilde\widetilde
\let\bar\overline

\newcommand{\ZZ}{\mathbb{Z}}
\newcommand{\CC}{\mathbb{C}}

\newcommand{\AAA}{\mc{A}}

\newcommand{\mc}[1]{\mathcal{#1}}
\newcommand{\DD}{\mc{D}}
\newcommand{\K}{\mc{K}}
\newcommand{\MM}{\mc{M}}
\newcommand{\MMM}{\mathfrak{M}}
\newcommand{\witt}{\mathfrak{witt}}
\newcommand{\switt}{\mathfrak{switt}}

\newcommand{\vir}{\mathfrak{vir}}
\newcommand{\ns}{\mathfrak{ns}}

\newcommand{\OO}{\mc{O}}

\def\co{\colon\thinspace}
\usepackage{colonequals}
\def\coeq{\colonequals}

\newcommand{\dd}[1]{\frac{d}{d{#1}}}

\newcommand{\Gr}{{\mathup{Gr}}}

\DeclareMathOperator{\dGr}{\Gr_2 \times \Gr_1}
\DeclareMathOperator{\sdGr}{\Gr_{3/2} \times \Gr_{1/2}}

\newcommand{\HHom}{\mathbb{H}\mathup{om}}

\DeclareMathOperator{\sym}{sym_1}

\DeclareMathOperator{\Ber}{Ber}

\DeclareMathOperator{\tr}{tr}

\DeclareMathOperator{\str}{str}

\newcommand{\dett}{{\mathup{det}}}
\renewcommand{\det}{\dett}

\definecolor{mygreen}{RGB}{80,150,10}
\newcommand{\old}[1]{}



%% file: standard_packages.tex
\usepackage[linktoc=page,unicode,naturalnames]{hyperref}
\usepackage{xurl}
\hypersetup{breaklinks=true}
\usepackage{amsmath,amssymb,amsthm}
\usepackage{mathtools}
\usepackage{graphicx,import,pdfpages}
\usepackage{tikz}
\usepackage{multicol}
\usepackage{tikz-cd}
\usetikzlibrary{arrows}
\usepackage[normalem]{ulem}
\usepackage{mathtools}
\usepackage{comment}
\usepackage{color} 

\makeatletter
\DeclareRobustCommand*{\tuburl}{\hyper@normalise\tuburl@}
\ExplSyntaxOn
\def\tuburl@#1
{
	\str_set:Nn\l_tmpa_str{#1}
	\str_remove_once:Nn\l_tmpa_str{https://}  
	\str_remove_once:Nn\l_tmpa_str{http://} 
	\expandafter\hyper@linkurl\expandafter{\expandafter\Hurl\expandafter{\l_tmpa_str}}{https://\l_tmpa_str}
}
\ExplSyntaxOff
\makeatother
\usepackage[shortlabels]{enumitem}

\usepackage{soul}

\usepackage[capitalize]{cleveref}
\crefname{equation}{}{}

\usepackage{arydshln}